\documentclass[11pt]{amsart}

\title[Arens regularity of projective tensor products]{
 Arens regularity of projective tensor products
}
\usepackage[ansinew]{inputenc}
\usepackage{yfonts}
\usepackage{calligra}
\usepackage{amscd}
\usepackage{amssymb, latexsym}
\usepackage{mathrsfs}
\usepackage{fancyhdr}
\usepackage{enumerate}
\newcommand{\oop}{\widehat\otimes}

\usepackage{textcomp, phonetic}
\usepackage{setspace}
\usepackage{latexsym}

\usepackage[all]{xy}
\usepackage{tikz}
\theoremstyle{plain}
\newtheorem{thm}{\sc Theorem}[section]
\newtheorem{cor}[thm]{\sc Corollary}
\newtheorem{lem}[thm]{\sc Lemma}

\newtheorem{prop}[thm]{\sc Proposition}

\newcommand{\ot}{\otimes}

\newcommand{\C}{\mathbb{C}}
\newcommand{\N}{\mathbb{N}}

\newcommand{\ds}{\displaystyle}
\newcommand{\al}{\alpha}
\newcommand{\bt}{\beta}

\newenvironment{pf}{\noindent {\sc Proof:}}{\hfill $\Box$}
\begin{document}

\author[ V. Rajpal And A. Kumar]{
 VANDANA RAJPAL AND AJAY KUMAR$^{1}$}
\address{Department of Mathematics\\
University of Delhi\\
Delhi\\
India.}
\email{akumar@maths.du.ac.in}
\address{Department of Mathematics\\
University of Delhi\\
Delhi\\
India.}
\email{vandanarajpal.math@gmail.com}
\thanks{}
\footnotetext[1]{- corresponding author}

\maketitle

\begin{abstract}
For completely contractive Banach algebras $A$ and $B$ (respectively operator algebras $A$ and $B$), the necessary and sufficient conditions for  the operator space projective tensor
product $A\widehat{\otimes}B$ (respectively the Haagerup tensor product $A\otimes^{h}B$) to be Arens regular are obtained.
Using the non-commutative Grothendieck's inequality, we  show  that, for $C^*$-algebras $A$ and
$B$, the Arens regularity of Banach algebras
 $A\otimes^{h}B$, $A\ot^{\gamma} B$, $A\ot^{s} B$ and $A\widehat{\otimes}B$ are equivalent, where $\otimes^h$, $\otimes^{\gamma}$, $\ot^s$ and $\widehat{\otimes}$ are  the Haagerup, the Banach space projective tensor norm, the Schur tensor norm   and
the operator space projective tensor norm, respectively.
\end{abstract}
\maketitle
\section{Introduction}
For a Hilbert space $H$, let $B(H)$ denote the space of all bounded operators on $H$. An operator space $X$ on $H$ is  a closed subspace of
$B(H)$.
For a operator spaces $X$ and $Y$, and $u$ an element in the algebraic tensor product $X\otimes Y$,  the operator space projective tensor norm is defined to be\\
 \hspace*{2.3 cm }$\|u\|_{\wedge}=\inf\{\|\alpha\|\|x\|\|y\|\|\beta\|:u=\alpha(x\otimes y)\beta\},$\\
where $\alpha\in M_{1,pq}$, $\beta\in M_{pq,1}$, $x\in M_{p}(X) $ and $y\in M_{q}(Y)$, $p,q\in \mathbb{N}$, and
$x\otimes y=(x_{ij}\otimes y_{kl})_{(i,k),(j,l)}\in M_{pq}(X\otimes Y)$.
The normed space $ (X\ot Y, \|\cdot\|_{\wedge})$ will be  denoted by $X\ot_{\wedge} Y$ and the
completion of $X\ot_{\wedge} Y$ is denoted by $X\widehat{\ot} Y$, known as the operator space
projective tensor product of $X$ and $Y$. The Haagerup norm on the algebraic tensor product of
two operator spaces $X$ and $Y$ is defined, for $u\in X\otimes Y$, by
$\|u\|_{h}=\inf\{\|x\| \|y\|\}$,  where infimum is taken over all the expressions $ u=x\odot y=\ds\sum_{k=1}^r x_{1k}\ot y_{k1}$, where
$x\in M_{1,r}(X), y\in M_{r,1}(Y), r\in\mathbb{N} $. The Haagerup tensor product $X\otimes^{h}Y$ is defined to be the completion of $X\otimes Y$
in the norm $\|\cdot\|_{h}$.
%

It is well known that the Haagerup tensor norm is
injective, associative, functorial and projective, and may be used to linearize the complete bounded bilinear forms, that is,  $CB(X\times Y , \C)\cong (X\otimes^{h} Y)^*$, $CB(X\times Y , \C)$ denotes the collection of complete bounded bilinear forms on $X\times Y$. A bilinear form
  $\phi: X\times Y \to \C$ is said to be
completely bounded if $ \|\phi\|_{cb}:=\displaystyle\sup_{n}\|\phi_n\| < \infty$, where the $n^{th}$
amplification is a map $\phi_{n}:M_n(X)\times M_{n}(
Y)
\to M_{n}(\C)$ defined by specifying the $(i,j)$ entry of
$\phi_n \big( (x_{ij}), (y_{ij}))$ to be $\ds\sum_{k} \phi(x_{ik}, y_{kj})$. The operator space projective
tensor norm is projective, functorial, associative and symmetric. But it is not injective and may be used to linearize the jointly completely bounded
bilinear forms,  that is,  $JCB(X\times Y , \C)\cong (X\oop Y)^*$.
 $JCB(X\times Y , \C)$ denotes the Banach space of  jointly completely bounded bilinear forms, which
becomes an operator space by identifying $M_n(JCB(X\times
Y ,\C))$ with $JCB(X \times Y ,M_n(\C))$, for each $n \in \N$, and a bilinear form $\varphi: X \times Y
 \to \C$ is said to be  jointly
completely bounded if the associated maps $\varphi_{n}:M_n(X)\times M_{n}(Y)\to M_{n^2}(\C)$ given by
\begin{equation}
\varphi_n \big( (x_{ij}), (y_{kl}) \big )=\big(\varphi(x_{ij},y_{kl})\big), \, \, n \in \mathbb{N}\notag
\end{equation}
are uniformly bounded, and in this case we denote $\|\varphi\|_{jcb} = \displaystyle\sup_{n}\|\varphi_n\|$, see \cite{effros}, \cite{pisi} for the development of tensor product of operator spaces.

In 1951, Arens showed that the second dual space $A^{**}$ of a Banach algebra $A$ admits two Banach algebra
products known as the first  and second  Arens products.  Each of
these products extends the original multiplication on $A$ when
$A$ is canonically imbedded in its second dual $A^{**}$. In this  note, we wish to draw attention  when the two Arens products agree
on the second dual of $A\widehat{\ot} B$ ( $A\ot^{h} B$) for completely contractive Banach algebras $A$ and $B$ (for operator algebras $A$ and
$B$).  It is shown that for completely contractive Banach algebras  $A$ and $B$ (for operator algebras $A$ and $B$), $A\widehat{\otimes} B$ ($A\otimes^{h} B$) is Arens regular if and only if every jointly completely bounded bilinear form $m:A\times B\to \mathbb{C}$ (resp. every completely bounded bilinear form
$m:A\times B\to \C$) is  biregular. This is then used to show an astonishing fact that, for $C^*$-algebras $A$ and $B$, the Arens regularity of
 $A\widehat{\otimes}B$ and $A\otimes^{h}B$ (resp. $A\ot^\gamma B$) is equivalent. Furthermore, for exact operator algebras $V$ and $W$, the Arens regularity
of $V\ot^h W$ and $V\widehat{\ot} W$ is shown to be equivalent.





\section{Arens Regularity of $A\widehat{\ot} B$}

For an operator space $X$, a closed subspace $\tilde{X}$ of $X$ is said to be
completely complemented if there exists a completely bounded (cb) projection $P$ from $X$ onto $\tilde{X}$.

We begin by stating a lemma which follows easily using the functorial property of operator space projective tensor product.

\begin{lem}\label{sy6}
Let $\tilde{X}$, $\tilde{Y}$ be completely  complemented subspaces of the operator spaces $X$ and $Y$ complemented by cb
projection having  cb norm equal to 1, respectively. Then $\tilde{X}\widehat{\otimes} \tilde{Y}$ is a closed subspace of $X\widehat{\otimes} Y$.
\end{lem}
For any normed space $X$,  $X_{1}$  and $X_{1}^{\text{o}}$ will denote the closed unit ball  and  the open unit ball of $X$, respectively. For
normed space $X$ and $Y$, the
normed linear space obtained by equipping $X\ot Y$ with $\|\cdot\|_{\al}$ norm is denoted
by $X \ot_{\al} Y$, and the completion of $X\ot_{\al}Y$ is denoted by $X\ot^{\al}Y$.


%
The second dual $A^{**}$ of  a Banach algebra $A$ possesses two natural
Banach algebra products denoted by $\square $ and $\lozenge$. We briefly recall the definition of these products.
  For $F, G\in A^{**}$ and $f\in A^{*}$, the two multiplications on $A^{**}$ are given by\\
\hspace*{3 cm} $F \square G(f)=F(_{G}f)$ and $F \lozenge G(f)=G(f_{F})$.\\
Here $_{G}f$, $f_{F}$ are the elements  of $A^{*}$ defined by  $_{G}f(b)=G(fb)$ and $f_{F}(b) =F( bf)$ for any $b\in A$,
whereas $bf \in A^{*}$, $f b \in A^{*}$ given by $bf(a)=f(ab)$ and $f b(a)=f(ba)$ for any $a\in A$.

We say that $A$ is   Arens regular if the two products coincide, i.e.
\\
\hspace*{3 cm}	$F \square G= F \lozenge G$ for all $F, G\in A^{**}$. \\
 A functional $f \in A^*$ is said to be wap (weakly almost periodic) on $A$ if
the set $\{af:\;\;a\in A_1\}$ is relatively weakly compact (rwc).  This is equivalent to saying that the bounded operator   $L_f:A\to A^*$, defined
by  $L_f(a)=fa$,
 is weakly compact or the following `Double Limit Criterion'\cite{groth}  holds:
for any two nets $(a_i)_{i}$, $(b_j)_{j}$ in $A_1$,
\[\displaystyle\lim_{i}\displaystyle\lim_{j} f(a_{i}b_{j})=\displaystyle\lim_{j}\displaystyle\lim_{i}f(a_{i}b_{j})\]
whenever both iterated limits exist. Let $wap(A)$ be a set  consisting of all weakly almost periodic  functionals on
$A$.  It is known that $A$ is Arens regular
if and only if $wap(A) = A^*$\cite{pym}. Every operator algebra,
in particular, every $C^*$-algebra, is Arens regular. By (\cite{yood}, Corollary 6.4) and the fact that the Banach space projective tensor norm $(\ot^{\gamma}) $and the operator space projective norm $(\widehat{\ot})$ on tensor product of two $C^*$-algebras is equivalent if and only if one of the $C^*$-algebras is subhomogeneous, it follows that  the Arens regularity of  $A\ot^{\gamma} B$ and $A\widehat{\ot} B$ is equivalent in this case. In this paper, we show that
the result is true in general. In particular, for a compact infinite Hausdorff group $G$, $C(G)\widehat{\otimes} C(G)$ ($C(G)\ot^h C(G)$) is not Arens regular Banach algebras as $C(G)\otimes^{\gamma} C(G)$ is not \cite{error}, $C(G)$ being the commutative $C^*$-algebra of continuous functions on $G$.

A bilinear form $m:A\times B\to \mathbb{C}$ is said to be biregular
 if for any two pairs of sequences $(a_{i})$, $(b_{j})$ in $A_{1}$ and $(c_{i})$, $(d_{j})$ in $B_{1}$,
we have $\displaystyle\lim_{i}\displaystyle\lim_{j} m(a_{i}b_{j}, c_{i}d_{j})=\displaystyle\lim_{j}\displaystyle
\lim_{i}m(a_{i}b_{j}, c_{i}d_{j})$, provided that these limits exist. It was shown in \cite{ulger} that the Banach space projective tensor product $A\ot^{\gamma} B$ is Arens regular if and
only if every bounded bilinear form is biregular. We now prove its analogous result for the operator space projective and the Haagerup tensor
product.

It is known that  $A\widehat{\ot}B$ and $A \ot^h B$ are
 Banach algebras  for  operator algebras $A$ and $B$ with respect to natural multiplication $(a\ot b)(c\ot d)=ac \ot bd$ for all $a,c\in A$ and $b,d\in B$
~\cite{r3}, \cite{blect}. One can use the same technique given therein for completely contractive Banach algebras  so as to obtain the following:

\begin{prop}\label{anin4}
For any completely contractive Banach algebras $A$ and $B$, $A\widehat{\ot} B$
is a completely contractive Banach algebra, and it is a Banach $^*$-algebra under the natural involution provided both $A$
and $B$ have isometric involution.
\end{prop}
\begin{prop}\label{arin55}
For completely contractive Banach algebras  $A$ and $B$, $A\widehat{\otimes} B$  is Arens regular if and only if every jointly completely bounded bilinear form $m:A\times B\to \mathbb{C}$  is  biregular.
\end{prop}
\begin{proof}
Assume that $A\widehat{\otimes} B$ is Arens regular. Let $m:A\times B\to \mathbb{C}$  be a jointly completely bounded bilinear form
such that  for any two pairs of sequences $(a_{i})$, $(b_{j})$ in $A_{1}$ and $(b_{i})$, $(d_{j})$ in $B_{1}$, the iterated limits
$\displaystyle\lim_{i}\displaystyle\lim_{j} m(a_{i}b_{j}, c_{i}d_{j})$ and $\displaystyle\lim_{j}\displaystyle\lim_{i}m(a_{i}b_{j}, c_{i}d_{j})$
exist. Since $(A\widehat{\otimes} B)^* = JCB(A\times B, \mathbb{C})$~\cite{effros}, so there exists $f\in (A\widehat{\otimes} B)^{*}$ such that
$f(a\otimes b)=m(a,b)$ for all $a\in A$,  $b\in B$.  Clearly, both the sequence $(a_{i}\otimes c_{i})$,
$(b_{j}\otimes d_{j})$ belong to the closed unit ball of $A\widehat{\otimes} B$. So $\displaystyle\lim_{i}\displaystyle\lim_{j} f(a_{i}b_{j}\otimes c_{i}d_{j})=\displaystyle\lim_{j}\displaystyle\lim_{i}f(a_{i}b_{j}\otimes c_{i}d_{j})$ by the assumption and hence $\displaystyle\lim_{i}\displaystyle\lim_{j} m(a_{i}b_{j}, c_{i}d_{j})=\displaystyle\lim_{j}\displaystyle\lim_{i}m(a_{i}b_{j}, c_{i}d_{j})$.

For the converse,  let $f$ be a  linear functional on
$A\widehat{\otimes} B$ and $m:A\times B\to \mathbb{C}$ be the corresponding jointly completely bounded bilinear form.  Note that the closed unit ball of $A\widehat{\ot}B$ is the closure of the set $\{\alpha(a\ot b)\beta: \alpha\in M_{1,n^2}, \beta\in M_{n^2,1}, a\in M_n(A)_1, b\in M_n(B)_1, n\in \mathbb{N}\}$, and  $L_{f}((A\widehat{\otimes} B)_{1})\subseteq \text{cl}(\text{co}H(f))$, where $H(f)=\{f .\al( a \otimes b) \beta: \al\in M_{1,n^2}, a\in M_n(A)_1\; \text{and} \;b\in M_n(B)_{1}, n\in \mathbb{N}\}$.
 By Krein-Smulyan theorem, it suffices to show that $H(f)$ is relatively weakly compact. By  (\cite{ulger}, Lemma 3.3), $H(f)$ is
 relatively weakly compact if and only if for any two sequences $(f. \al_i (a_i  \otimes b_i)\bt_i)_{i\in I}$ and
$(\tilde{\al}_j( \tilde{a}_j  \otimes \tilde{b}_j)\tilde{\bt}_j)$, we have $\displaystyle\lim_{i}\displaystyle\lim_{j}  f.\al_i (a_i  \otimes b_i)\bt_i (\tilde{\al}_j(\tilde{a}_j
\otimes \tilde{b}_j) \tilde{\bt}_j)=\displaystyle\lim_{j}\displaystyle\lim_{i} f.\al_i (a_i  \otimes b_i)\bt_i (\tilde{\al}_j(\tilde{a}_j
\otimes \tilde{b}_j) \tilde{\bt}_j)$,
provided that these limits exist; i.e. we  must have $\displaystyle\lim_{i}\displaystyle\lim_{j} f(\ds\sum_{k,l, m,n} \al_{1,km}^i (x_{kl}^i \ot y_{mn}^i) \bt_{ln,1} ^i\ds\sum_{p,q, o,t} \tilde{\al}_{1,po}^j (x_{pq}^j \ot y_{ot}^j) \tilde{\bt}_{qt,1}^j )
$\\ $=\displaystyle\lim_{j}\displaystyle\lim_{i} f(\ds\sum_{k,l, m,n} \al_{1,km}^i (x_{kl}^i \ot y_{mn}^i) \bt_{ln,1} ^i\ds\sum_{p,q, o,t} \tilde{\al}_{1,po}^j (x_{pq}^j \ot y_{ot}^j) \tilde{\bt}_{qt,1}^j )$ provided that these limits exist.  So in terms of jcb bilinear form $m$, we must have\\$\displaystyle\lim_{i}\displaystyle\lim_{j} \ds\sum_{k,l, m,n} \ds\sum_{p,q, o,t} \al_{1,km}^i  \tilde{\al}_{1,po}^j m(x_{kl}^i x_{pq}^j,  y_{mn}^i  y_{ot}^j) \bt_{ln,1} ^i   \tilde{\bt}_{qt,1}^j $ \\ $=\displaystyle\lim_{j}\displaystyle\lim_{i}  \ds\sum_{k,l, m,n} \ds\sum_{p,q, o,t} \al_{1,km}^i  \tilde{\al}_{1,po}^j m(x_{kl}^i x_{pq}^j,  y_{mn}^i  y_{ot}^j) \bt_{ln,1} ^i  \tilde{\bt}_{qt,1}^j  $ provided limits exist. Since each bounded net has a convergent subnet, so corresponding to $(\al_{1,km}^{i})_{i\in I}$ (resp. $(\tilde{\al}_{1,po}^{j})_{j\in J}$) we have  $ (\al_{1,km}^{i_a})_{i_a\in I_1}$ ( resp., $(\tilde{\al}_{1,po}^{j_c})_{j_c\in J_1}$), $I_1 \subseteq I$ (resp., $J_1\subseteq J$). Now if we look at $(\bt_{ln,1} ^{i_a})_{i_a \in I_1}$ (resp. $(\tilde{\bt}_{qt,1}^{j_c})_{j_c \in J_1}$) then this net has  a convergent subnet say  $(\bt_{ln,1} ^{i_b})_{i_b\in I_2}$ (resp. $(\tilde{\bt}_{qt,1}^{j_d})_{j_d\in J_2} $), $I_2\subseteq I_1$ ( $J_2\subseteq J_1$). Continuing in this way, we get   indices such that  $\displaystyle\lim_{j_h}\displaystyle\lim_{i_g} m(x_{kl}^{i_g} x_{pq}^{j_h},  y_{mn}^{i_g}  y_{ot}^{j_h})$ and $\displaystyle\displaystyle\lim_{i_g}\lim_{j_h} m(x_{kl}^{i_g} x_{pq}^{j_h},  y_{mn}^{i_g}  y_{ot}^{j_h})$ exist. Now note that $ \displaystyle\displaystyle\lim_{i_g}\lim_{j_h} \ds\sum_{k,l, m,n} \ds\sum_{p,q, o,t} \al_{1,km}^{i_g}  \tilde{\al}_{1,po}^{j_h}  m(x_{kl}^{i_g} x_{pq}^{j_h},  y_{mn}^{i_g}  y_{ot}^{j_h})  \bt_{ln,1} ^{i_g}  \tilde{\bt}_{qt,1}^{j_h}$ \\$ =     \ds\sum_{k,l, m,n} \ds\sum_{p,q, o,t} \displaystyle\displaystyle\lim_{i_g}\lim_{j_h} \al_{1,km}^{i_g}  \tilde{\al}_{1,po}^{j_h}  m(x_{kl}^{i_g} x_{pq}^{j_h},  y_{mn}^{i_g}  y_{ot}^{j_h})  \bt_{ln,1} ^{i_g}  \tilde{\bt}_{qt,1}^{j_h}$\\$=  \displaystyle\displaystyle\lim_{j_h}\lim_{i_g} \ds\sum_{k,l, m,n} \ds\sum_{p,q, o,t} \al_{1,km}^{i_g}  \tilde{\al}_{1,po}^{j_h}  m(x_{kl}^{i_g} x_{pq}^{j_h},  y_{mn}^{i_g}  y_{ot}^{j_h})  \bt_{ln,1} ^{i_g}  \tilde{\bt}_{qt,1}^{j_h} $ by the biregularity of $m$. Let $S=: \displaystyle\lim_{i}\displaystyle\lim_{j} \ds\sum_{k,l, m,n} \ds\sum_{p,q, o,t} \al_{1,km}^i  \tilde{\al}_{1,po}^j m(x_{kl}^i x_{pq}^j,  y_{mn}^i  y_{ot}^j) \bt_{ln,1} ^i   \tilde{\bt}_{qt,1}^j $ and $T=:  \displaystyle\displaystyle\lim_{j} \lim_{i} \ds\sum_{k,l, m,n} \ds\sum_{p,q, o,t} \al_{1,km}^i  \tilde{\al}_{1,po}^j m(x_{kl}^i x_{pq}^j,  y_{mn}^i  y_{ot}^j) \bt_{ln,1} ^i   \tilde{\bt}_{qt,1}^j $. For each $i$ and $j$, set $S_i= \displaystyle\lim_{j} \ds\sum_{k,l, m,n} \ds\sum_{p,q, o,t} \al_{1,km}^i  \tilde{\al}_{1,po}^j m(x_{kl}^i x_{pq}^j,  y_{mn}^i  y_{ot}^j) \bt_{ln,1} ^i   \tilde{\bt}_{qt,1}^j $ and $T_j=  \displaystyle \lim_{i} \ds\sum_{k,l, m,n} \ds\sum_{p,q, o,t} \al_{1,km}^i  \tilde{\al}_{1,po}^j m(x_{kl}^i x_{pq}^j,  y_{mn}^i  y_{ot}^j) \bt_{ln,1} ^i   \tilde{\bt}_{qt,1}^j $. Obviously, $T_{j_h}= \displaystyle \lim_{i_g} \ds\sum_{k,l, m,n} \ds\sum_{p,q, o,t} \al_{1,km}^{i_g}  \tilde{\al}_{1,po}^{j_h} m(x_{kl}^{i_g} x_{pq}^{j_h},  y_{mn}^{i_g}  y_{ot}^{j_h}) \bt_{ln,1} ^{i_g}   \tilde{\bt}_{qt,1}^{j_h} $, and   $\displaystyle\lim_{i_g}S_{i_g}=
\displaystyle\lim_{j_h}T_{j_h}$. Hence $S=T$.
\end{proof}
The proof of the following proposition is on the similar lines as above by noticing that the closed uint ball of $A\ot^h B$ is the closure of the set $\{  a \odot b: a\in M_{1,n}, b\in M_{n,1}, n\in \mathbb{N}\}$.
\begin{prop}\label{arin5655}
For  operator algebras $A$ and $B$, $A\otimes^{h} B$ is Arens regular if and only if every every completely bounded bilinear form
$m:A\times B\to \C$ is biregular.
\end{prop}

\begin{cor}
Suppose that the algebras $A$ and $B$ are not trivial and  that $A\widehat{\otimes} B$  ($A\otimes^{h} B$) is Arens regular, then $A$ and $ B$ are Arens regular.
\end{cor}
\begin{proof}
Suppose that $f\in A^{*}$ and $(a_{i})$, $(c_{j})$ are sequences in $A_{1}$ such that
$\displaystyle\lim_{i}\displaystyle\lim_{j} f(a_{i}c_{j})$ and
$\displaystyle\lim_{j}\displaystyle\lim_{i}f(a_{i}c_{j})$ exist.
Since $B$ is not trivial, there exist $b\neq 0$, $b'\neq 0$ such that $bb'\neq 0$. Using Hahn Banach theorem, select $g\in B^{*}$ such that $g(bb')=1 $. Now define, $f\otimes g:A\otimes B\to \mathbb{C}$ as $f\otimes g(a\otimes b)=f(a)g(b)$.
Then for $x=\displaystyle\sum_{i=1}^{n}a_{i}\otimes b_{i}\in A\otimes B$, $|f\otimes g(x)|\leq \|f\|\|g\|\|x\|_{\lambda}\leq
\|f\|\|g\|\|x\|_{\wedge}$, so that $f\otimes g$ can be extended to the continuous linear functional
on $A\widehat{\otimes} B$. Let the extension be denoted by $f\widehat{\otimes} g$.  Now, let $b_{i}=b$ and $d_{j}=b'$
 for all $i$, $j$. Then  $\displaystyle\lim_{i}\displaystyle\lim_{j} f\widehat{\otimes}
g(a_{i}c_{j}\otimes b_{i}d_{j})=\displaystyle\lim_{j}\displaystyle\lim_{i}f\widehat{\otimes} g(a_{i}c_{j}\otimes b_{i}d_{j})$,
by Theorem \ref{arin55}. Hence $\displaystyle\lim_{i}\displaystyle\lim_{j} f(a_{i}c_{j})=\displaystyle\lim_{j}\displaystyle\lim_{i}f(a_{i}c_{j})$, so $A$ is Arens regular. Similarly, $B$ is Arens regular.
\end{proof}


We are now ready to present the main results.

\begin{thm}\label{main}
For completely contractive Banach algebras $A$ and $B$,  if $A\widehat{\otimes} B$ is Arens regular then  $A\otimes^{h} B$ is. Conversely,
for $C^{*}$-algebras $A$ and $B$, if  $A\otimes^{h} B$  is Arens regular so is $A\widehat{\otimes} B$.
\end{thm}
\begin{pf}
First part follows directly by Proposition \ref{arin55} and  the fact that   for $f\in (A\ot^h B)^*$, $f\circ i(a\ot b)=f(a\ot b)$, $a\in A, b\in B$, where  $i$ is a contractive
homomorphism map from $A\widehat{\otimes} B$ into $A\otimes^{h} B$ (\cite{effros}, Theorem 9.2.1).

For the converse, we first claim  that, for  $C^{*}$-algebra $A $, if $A\otimes^{h} A$  is Arens regular then $A \widehat{\ot} A$ is.
Assume that  $f\in JCB(A\times A,\C)$ be  such that
 for sequences $(a_{i})$, $(b_{j})$ in $A_{1}$ and $(c_{i})$, $(d_{j})$ in $A_{1}$, the
 iterated limits $\displaystyle\lim_{i}\displaystyle\lim_{j} f(a_{i}b_{j}, c_{i}d_{j})$ and $\displaystyle\lim_{j}\displaystyle\lim_{i}f
 (a_{i}b_{j}, c_{i}d_{j})$ exist. By (\cite{musat}, Lemma 3.1), $f$ can be decomposed as $f=f_1+f_2$, where $f_1$ and $f_2$
  are bounded  bilinear forms with $\|f_1\|_{cb}\leq \|f \|_{jcb}$
and $\|f_2^t\|_{cb}\leq \|f \|_{jcb}$, where $f_2^t(b, a) = f_2(a, b)$ for all $a,b\in A$. Since $(a_i)$ is a bounded net, so has a convergent subnet say $(a_{i_k})$. Consider $(c_{i_k})$ which is bounded so has a convergent subnet $(c_{i_l})$. Similarly, $(b_j)$ has a convergent subnet $(b_{j_m})$, and $(d_{j_m})$ has a further convergent subsequence $(d_{j_t})$. As $f_1$ is a bounded bilinear form,
therefore the iterated limits $\displaystyle\lim_{i_l}\displaystyle\lim_{j_t} f_1(a_{i_l}b_{j_t}, c_{i_l}d_{j_t})$
 and  $\displaystyle\lim_{j_t}\displaystyle\lim_{i_l}f_1
 (a_{i_l}b_{j_t}, c_{i_l}d_{j_t})$ exist. Since $f_2=f-f_1$, so we can assume that $\displaystyle\lim_{i_l}\displaystyle\lim_{j_t} f_2
(a_{i_l}b_{j_t}, c_{i_l}d_{j_t})$
 and  $\displaystyle\lim_{j_t}\displaystyle\lim_{i_l}f_2
 (a_{i_l}b_{j_t}, c_{i_l}d_{j_t})$ also exist. Therefore, $\displaystyle\lim_{i_l}\displaystyle\lim_{j_t} f(a_{i_l}b_{j_t}, c_{i_l}d_{j_t})=
\displaystyle\lim_{i_l}\displaystyle\lim_{j_t} f_1(a_{i_l}b_{j_t}, c_{i_l}d_{j_t})+\displaystyle\lim_{i_l}\displaystyle\lim_{j_t}f_2
 (a_{i_l}b_{j_t}, c_{i_l}d_{j_t})$,
 which is further equal to $\displaystyle\lim_{i_l}\displaystyle\lim_{j_t} f_1(a_{i_l}b_{j_t}, c_{i_l}d_{j_t})+\displaystyle\lim_{i_l}
\displaystyle\lim_{j_t}f_2^t
 (c_{i_l}d_{j_t}, a_{i_l}b_{j_t})$. Now by the Arens regularity of $A\otimes^{h} A$ we have $\displaystyle\lim_{i_l}\displaystyle\lim_{j_t}
 f(a_{i_l}b_{j_t}, c_{i_l}d_{j_t})=\displaystyle\lim_{j_t}\displaystyle\lim_{i_l} f(a_{i_l}b_{j_t}, c_{i_l}d_{j_t})$.
Let $\alpha=: \displaystyle\lim_{i}\displaystyle\lim_{j} f(a_{i}b_{j}, c_{i}d_{j})$ and $\beta=:\displaystyle\lim_{j}\displaystyle\lim_{i}f
 (a_{i}b_{j}, c_{i}d_{j})$. For each $i$ and $j$, set $\alpha_i=\displaystyle\lim_{j} f(a_{i}b_{j}, c_{i}d_{j})$
and $\beta_j=\displaystyle\lim_{i}f
 (a_{i}b_{j}, c_{i}d_{j})$. Since $(f(a_{i_l}b_{j}, c_{i_l}d_{j}))_{l}$ is a sequence which converges to $\alpha_{i_l}$,
so every subsequence of it also converges to $\alpha_{i_l}$. Therefore,   we have $\displaystyle\lim_{i_l}\alpha_{i_l}=
\displaystyle\lim_{j_t}\beta_{j_t}$ and hence $\alpha=\beta$.

Now assume that $A\otimes^{h} B$  is Arens regular. So $((A\otimes^{h} B)\oplus (A\otimes^{h} B))_1$  is Arens regular
\cite{Arikan}, where $\oplus_1$ denotes the direct sum of the algebras with $\|(u,v)\|_1=\|u\|_h+\|v\|_h$. We now claim that
$(A\oplus B)_\infty\ot^h (A\oplus B)_\infty$
is Arens regular, where $\oplus_{\infty}$ being the direct product of algebras with $\|(a,b)\|_{\infty}=\max\{\|a\|,\|b\|\}$.
 Consider a  map  $\theta: ((A\otimes^{h} B)\oplus (A\otimes^{h} B))_1 \to
(A\oplus B)_\infty\ot^
h (A\oplus B)_\infty $ defined as $\theta(\ds\sum_{i=1}^\infty a_i\ot b_i, \ds\sum_{i=1}^\infty c_i\ot d_i )=\ds\sum_{i=1}^\infty (a_i, d_i)\ot (c_i,b_i) $,  where  $\{a_i\}_{i=1}^{\infty}$, $\{b_i\}_{i=1}^{\infty}$,
$\{c_i\}_{i=1}^{\infty}$ and $\{d_i\}_{i=1}^{\infty}$  are strongly independent~\cite{sinc}. We show that $\theta$ is a  continuous algebra homomorphism.
 For the continuity of $\theta$, consider  \\
\hspace*{0.1265 cm}
$\|\theta(\ds\sum_{i=1}^\infty a_i\ot b_i, \ds\sum_{i=1}^\infty c_i\ot d_i)\|_h= \|\ds\sum_{i=1}^\infty (a_i,d_i)\ot (c_i,b_i)\|_h$\\
\hspace*{4 cm} $ \leq
\|\ds\sum_{i=1}^\infty (a_ia_i^*, d_i
d_i^*)\|_{\infty}^{1/2} \|\ds\sum_{i=1}^\infty (c_i^*c_i, b_i^*
b_i)\|_{\infty}^{1/2}$\\
\hspace*{4 cm}
$ = \max \{\|\ds\sum_{i=1}^\infty a_ia_i^*\|,\|\ds\sum_{i=1}^\infty d_i
d_i^*\|\}^{1/2} \max \{\|\ds\sum_{i=1}^\infty b_i^*b_i\|,\|\ds\sum_{i=1}^\infty c_i^*
c_i\|\}^{1/2}$\\
\hspace*{4 cm} $ \leq \frac{1}{2} (\max \{\|\ds\sum_{i=1}^\infty a_ia_i^*\|,\|\ds\sum_{i=1}^\infty d_i
d_i^*\|\}+ \max \{\|\ds\sum_{i=1}^\infty b_i^*b_i\|,\|\ds\sum_{i=1}^\infty c_i^*
c_i\|\} )$.
Thus  $\|\ds\sum_{i=1}^\infty (a_i,d_i)\ot (c_i,b_i)\|_h  \leq  \frac{1}{2} (\|\ds\sum_{i=1}^\infty a_ia_i^*\|+\|\ds\sum_{i=1}^\infty d_i
d_i^*\| + \|\ds\sum_{i=1}^\infty b_i^*b_i\|+\|\ds\sum_{i=1}^\infty c_i^*
c_i\|)$. We can rewrite it as $\|\ds\sum_{i=1}^\infty (t^{1/2}a_i, t^{1/2}d_i)\ot (t^{-1/2}b_i,t^{-1/2}c_i)\|_h  \leq  \frac{1}{2}
 (t\|\ds\sum_{i=1}^\infty  a_ia_i^*\|+t \|\ds\sum_{i=1}^\infty d_i
d_i^*\| + t^{-1} \|\ds\sum_{i=1}^\infty b_i^*b_i\|+t^{-1}\|\ds\sum_{i=1}^\infty  c_i^*
c_i\|)$ for any $t>0$. Take infimum over $t>0$ and use the fact that $\ds\inf_{t>0}
\frac{t\alpha +t^{-1}\beta}{2} =\sqrt{\alpha \beta} $, we get  $\|\ds\sum_{i=1}^\infty (a_i,d_i)\ot (b_i,c_i)\|_h  \leq
 \|\ds\sum_{i=1}^\infty  a_ia_i^*\|^{1/2} \|\ds\sum_{i=1}^\infty b_i^*b_i\|^{1/2}+ \|\ds\sum_{i=1}^\infty d_i
d_i^*\|^{1/2} \|\ds\sum_{i=1}^\infty  c_i^*
c_i\|^{1/2}$. Thus  $\|\theta(\ds\sum_{i=1}^\infty a_i\ot b_i, \ds\sum_{i=1}^\infty c_i\ot d_i)\|_h \leq  \|\ds\sum_{i=1}^\infty a_i\ot b_i\|_h
+ \|\ds\sum_{i=1}^\infty c_i\ot d_i\|_h =\|(\ds\sum_{i=1}^\infty a_i\ot b_i, \ds\sum_{i=1}^\infty c_i\ot d_i)\|_1$. Hence $\theta $
 is continuous. One can easily verify that $\theta$ is an algebra homomorphism. Now let $f\in ((A\oplus B)_\infty\ot^h (A\oplus B)_\infty)^*$ be  such that
 for sequences $(a_{i}, b_i)$,  $(\tilde{a_j}, \tilde{b_{j}})$, $(a_{i}', b_i')$ and $(\tilde{a_j}', \tilde{b_{j}}')$ in $((A\oplus B)_\infty)_1$, the
 iterated limits $\displaystyle\lim_{i}\displaystyle\lim_{j} f((a_{i}\tilde{a_j},  b_i \tilde{b_{j}})\ot  (a_{i}'\tilde{a_{j}}', b_{i}' \tilde{b_{j}}'))$ and $\displaystyle\lim_{j}\displaystyle\lim_{i} f((a_{i}\tilde{a_j},  b_i \tilde{b_{j}})\ot  (a_{i}'\tilde{a_j}', b_i' \tilde{b_{j}}'))$ exist. Then the Arens regularity of $((A\otimes^{h} B)\oplus (A\otimes^{h} B))_1$ implies Arens regularity of $(A\oplus B)_\infty\ot^h (A\oplus B)_\infty$ by $f\circ \theta(a_{i} \tilde{a_j} \ot b_i' \tilde{b_{j}}', a_{i}'\tilde{a_j}' \ot b_i \tilde{b_{j}} )=  f((a_{i}\tilde{a_j},  b_i \tilde{b_{j}})\ot  (a_{i}'\tilde{a_{j}}', b_{i}' \tilde{b_{j}}'))$, and Proposition \ref{arin55}. Now consider the natural projections $P: (A\oplus B)_{\infty}\to A$  and $Q: (A\oplus B)_{\infty}\to B$ taking $(a,b)\to a$ and $(a,b)\to b$, respectively. By the definition of max norm, it follows that $P$ and $Q$ are completely bounded with $\|P\|_{cb}\leq 1$ and $\|Q\|_{cb}\leq 1$. Clearly, the above map $P$ is a completely
positive, and $P(a_1(a,b))=P((a_1,0)(a,b))=P(a_1a,0)=a_1a=a_1P(a,b)$ for $a_1,a\in A$ and $b\in B$. Therefore, $P$ is a conditional expectation
from  $(A\oplus B)_{\infty}$ onto $A$, and similarly $Q$ is. Thus $A\widehat{\ot} B$  is a closed subalgebra of $(A\oplus B)_{\infty}
\widehat{\ot} (A\oplus B)_{\infty}$ by Lemma \ref{sy6}, and hence
 $A\widehat{\ot} B$
is   Arens regular.
\end{pf}

For an amenable locally compact Hausdorff groups $G$ and $H$, $A(G)\widehat{\otimes} A(H)$ is completely isometrically isomorphic to  $A(G\times H)$ \cite{effr}, $A(G)\widehat{\otimes} A(H)$ is Arens regular
 if and only if  $G$ and $H$ are  finite by (~\cite{ALAU}, Proposition 3.3). In general, for a locally compact groups $G$ and $H$,
if $A(G)\widehat{\otimes} A(H)$ ($A(G)\ot^h A(H)$)
is Arens regular then $G$ and $H$ are discrete (~\cite{for1}, Theorem 3.2). Furthermore, $A(G)\widehat{\otimes} A(H)$ is not Arens regular if $G$ and $H$ contain an infinite abelian subgroup ~\cite{for2}.
In particular for an amenable locally compact Hausdorff groups $G$ and $H$, $A(G)\otimes^{h} A(H)$ is Arens regular
 if and only if  $G$ and $H$ are  finite.

For operator spaces $X$ and $Y$, we say that a sesquilinear form $\phi $ on $X\times Y$ is jointly completely bounded if $\|\phi\|_{jcb}:=
\sup\{\|[\phi(x_{ij},y_{kl})]\|
: \|[x_{ij}]\|_{M_n(X)}\leq 1, \|[y_{ij}]_{M_n(\overline{Y})}\|\leq 1\} < \infty$, where $\overline{Y}$ is the conjugate
of the operator space $Y$. It is known that,
for a  $C^*$-algebra $A$,  the complex conjugate $\overline{A} $ of operator space $A$  and the opposite $C^*$-algebra $A^{op}$ are $C^*$-isomorphic. Therefore,
for $C^*$-algebras $A$ and $B$, a sesquilinear form $\phi $ on $A\times B$ is jointly completely bounded if $\|\phi\|_{jcb}=
\sup\{\|[\phi(x_{ij},y_{kl})]\|
: \|[x_{ij}]\|_{M_n(A)}\leq 1, \|[y_{ij}]_{M_n(B^{op})}\|\leq 1\}< \infty$. For more details about the complex conjugate of an operator space
and the opposite operator spaces, the reader
may refer to \cite{pisi}.

\begin{lem}\label{ardd12}
For $C^*$-algebras $A$ and $B$, $A\widehat{\otimes} B$ is Arens regular if and only if
$A\widehat{\otimes} B^{op}$ is.
\end{lem}
\begin{pf}
Note that every  jointly completely bounded  bilinear form on $A\times B$ is
biregular if and only if  every sesquilinear jointly completely bounded  form on $A\times B$ is biregular. Indeed, suppose that  every  jointly completely bounded  bilinear form on $A\times B$ is
biregular, and take $\phi$ to be a sesquilinear jointly completely bounded bilinear form on $A\times B$ such that  for sequences $(a_{i})$, $(b_{j})$ in $A_{1}$ and $(c_{i})$, $(d_{j})$ in $B_{1}$, the
 iterated limits $\displaystyle\lim_{i}\displaystyle\lim_{j} \phi(a_{i}b_{j}, c_{i}d_{j})$ and
$\displaystyle\lim_{j}\displaystyle\lim_{i}\phi(a_{i}b_{j}, c_{i}d_{j})$ exist. Now define $\psi: A\times B \to \C$
as $\psi(a,b)=\phi(a,b^*)$ for all $a\in A$ and $b\in B$. Clearly, $\psi$ is a bilinear form. Consider $\psi_n([a_{ij}],[b_{kl}])=  (\psi(a_{ij},b_{kl}))=
(\phi(a_{ij},b_{kl}^*))=\phi_n([a_{ij}],[b_{kl}^*])$. Therefore,  \\
\hspace*{1 cm}$\|\psi_n([a_{ij}],[b_{kl}])\|=\|\phi_n([a_{ij}],[b_{kl}^*])\|\leq
\|\phi\|_{jcb}\|[a_{ij}]\|_{M_n(A)}\|[b_{kl}^*]\|_{M_n(B^{op})}$\\
\hspace*{3.5 cm} $\leq
\|\phi\|_{jcb}\|[a_{ij}]\|_{M_n(A)}\|[b_{lk}]^{*}\|_{M_n(B^{op})}$\\
\hspace*{3.6 cm}$= \|\phi\|_{jcb}\|[a_{ij}]\|_{M_n(A)}\|[b_{lk}]\|_{M_n(B^{op})}$,\\
\hspace*{3.5 cm} $=\|\phi\|_{jcb}\|[a_{ij}]\|_{M_n(A)}\|[b_{kl}]\|_{M_n(B)}$. \\
Thus $\psi$ is a jointly completely bounded bilinear form. Further, note that $\psi(a_ib_j, d_j^*c_i^*)=\phi(a_ib_j, c_id_j)$, and hence the claim.

In fact, from Proposition \ref{arin55}, this observation leads to the fact that  $A\widehat{\otimes} B$ is Arens regular if and only if
$A\widehat{\otimes} B^{op}$ is.
\end{pf}

\begin{thm}\label{main2}
For completely contractive Banach algebras  $A$ and $B$, if $A\otimes^{\gamma} B$ is Arens regular then
$A\widehat{\otimes} B$ is. Conversely, for $C^*$-algebras $A$ and $B$,
 if $A\widehat{\otimes} B$ is Arens regular then $A\otimes^{\gamma} B$ is.
\end{thm}
\begin{proof}
For $f\in (A\widehat{\ot} B)^*$, $f\circ i \in  (A\otimes^{\gamma} B)^*$ and $f(a\ot b)=f(a\ot b)$, $a\in A, b\in B$, where $i$ is a  canonical homomorphism from $A\otimes^{\gamma}B$  into $A\widehat{\otimes} B$, the Arens regularity of
$A\otimes^{\gamma} B$ implies that of  $A\widehat{\otimes} B$ using Proposition \ref{arin55} and (\cite{ulger}, Theorem 3.4).

For the converse, let  $m:A\times B\to \mathbb{C}$  be a bounded bilinear form such that  for sequences $(a_{i})$, $(b_{j})$ in $A_{1}$ and $(c_{i})$, $(d_{j})$ in $B_{1}$, the
 iterated limits $\displaystyle\lim_{i}\displaystyle\lim_{j} m(a_{i}b_{j}, c_{i}d_{j})$ and
$\displaystyle\lim_{j}\displaystyle\lim_{i}m (a_{i}b_{j}, c_{i}d_{j})$ exist. By  the non-commutative version of Grothendieck's inequality to
the setting of bounded bilinear forms on $C^*$-algebras, $m$ can be decomposed as $m=m_1+m_2$, where $m_1$ and $m_2$ are jointly completely
bounded bilinear forms on $A\times B$ and $A\times B^{op}$ respectively~\cite{shl}. Using Lemma \ref{ardd12} and  a similar argument  as in Theorem \ref{main}, we obtain the required result.
\end{proof}
From all the results above, the Arens regularity of all Banach algebras $A\ot^s B$, the Schur tensor product of $C^*$-algebras $A$ and $B$ \cite{vandee4},
 $A\ot^h B$, $A\ot^\gamma B$  and $A\widehat{\ot} B$ are equivalent.

By (\cite{ulgerr}, Theorem 7.6) and the above results, we have
\begin{cor}\label{arin6}
Let $A$ be a  unital $C^*$-algebra such that the von Neumann algebra
$A^{**}$ is not finite. Then the algebra $A\widehat{\otimes}A$ (resp. $A\otimes^{h}A$ ) is not Arens regular.
In particular, $A\widehat{\otimes}A^{**}$ (resp. $A\otimes^{h}A^{**}$) and $A^{**}\widehat{\otimes}A^{**}$ (resp. $A^{**}\otimes^{h}A^{**}$) are not Arens regular.
\end{cor}

In particular, for an infinite dimensional Hilbert space $H$, $B(H)\widehat{\otimes} B(H)$, $B(H)\widehat{\otimes} K(H)$ and $K(H)\widehat{\otimes} B(H)$ are not Arens regular
and so $B(H)\widehat{\otimes} K(H)+K(H)\widehat{\otimes} B(H)$ is not Arens regular by using
 (\cite{yood}, Corollary
6.3). Also, by (\cite{ulgerr}, Proposition 7.7 and Lemma 7.8), $K(H)\widehat{\otimes} K(H)$ is not Arens regular.

%

%

Recall that an operator space $X$ is exact if and only if it is locally embeds into a nuclear $C^*$-algebra (say $A$),
 i.e. there is  a constant $C$ such that
for any finite dimensional $E\subseteq  X$, there is  a subspace $\tilde{E} \subseteq A$  and an isomorphism $u: E\to \tilde{E}$ with
$\|u\|_{cb}\|u^{-1}\|_{cb}\leq
C$. Now using this definition  of exact operator space and the fact that direct sum of two nuclear $C^*$-algebras is nuclear
if and only if each one of them is, one can easily verify
that if $V$ and $W$ are exact operator algebras then $V\oplus W$ with sup-norm  is also exact.

Using the same idea as in Theorem \ref{main} and appealing to the non-commutative version of Grothendieck's inequality to
the setting of jointly completely bounded bilinear forms on exact operator spaces (\cite{shl}, Theorem 0.4), we have
\begin{prop}\label{main3}
For exact operator algebras $V$ and $W$, the Arens regularity of   $V\widehat{\otimes} W$ and $V\ot^h W$ is equivalent.
\end{prop}

\begin{prop}
Let $A$ and $B$ be any operator algebras such that every weakly compact operator from $A$ to $B^*$ is compact. Then $B\ot^h A$
is Arens regular. Conversely, assume that, for each $a \in A$ and $b \in B$, one of the left multiplication operators
$_a\tau$ or $_b\tau$ is compact and the other is weakly compact. Then every weakly compact operator from $A$ to $B^*$ is compact.
\end{prop}
\begin{proof}
Let $m$ be  a completely bounded bilinear form and $\tilde{m}: A\to B^*$ be the corresponding operator given by $\tilde{m}(a)(b)=m(a,b)$. By (\cite{effros}, Lemma 13.3.1),
 $\tilde{m}$ can be factored through a column Hilbert space so is weakly compact. Hence  $\tilde{m}$ is compact by hypothesis. Therefore,
by (\cite{ulger}, Theorem 4.5), $m$ is biregular and so $B\ot^h A$ is Arens regular. Converse part follows from (\cite{ulger2nd}, Theorem 5.3).
\end{proof}
In particular, for any operator algebra  $B$ for which $B^*$ has schur property, $B\ot^h A$
is Arens regular for any operator algebra $A$. Thus, for a compact Hausdorff group $G$, $C^*(G)\ot^h A$ is Arens regular for any
operator algebra $A$ by (\cite{ulger2nd}, Theorem 4.5). Similarly, for a compact dispersed topological space $K$, it follows from
(\cite{ulger2nd},  Theorem 6.4 and Corollary 6.5) that $C(K)\ot^h A$ is
Arens regular  for any operator algebra $A$.  Also, by (\cite{free}, Theorem 3.3), for a closed subalgebra $B$ of $K(H)$
having Dunford-Pettis property, $B\ot^h A$ is Arens regular for any operator algebra $A$.

We now  investigate the certain multiplication on one of the Banach algebras and look at the Arens regularity of the operator space projective,
the Banach space projective  and the Haagerup
tensor product.

One can immediately verify that if the multiplication on one of the Banach algebras is trivial, then  $A\widehat{\ot} B$ (resp. $A\ot^{h} B$)
is  always Arens regular.  Now let $B$ be any operator space. Define an algebra multiplication on $B$ as
 \[b_1b_2 = \phi(b_1)b_2 (b_1; b_2 \in B) ;\]
where $\phi\in B^*$ with $\|\phi\|\leq 1$. Note that such a $\phi$ is automatically a multiplicative linear functional and multiplication will be non-trivial provided
$\phi$ is one-to-one. Also, one can easily verify that $B$ is an associative Banach algebra under the above multiplication. Now by using
Ruan's axiom and the fact that $\|\phi\|_{cb}=\|\phi\|$, we have
\begin{align}
\|[b_{ij}][b'_{ij}]\|&=\|[\ds\sum_{k=1}^n\phi(b_{ik})b'_{kj}]\|\notag\\
&=\|[\phi(b_{ij})] [b'_{ij}]\| \notag\\
&\leq \|[\phi(b_{ij})]\|\|[b'_{ij}]\|\notag\\
&\leq \|[b_{ij}]\|\|[b'_{ij}]\| .\notag
\end{align}
Thus $B$ is an operator algebra by (\cite{quob}, Theorem 1.3).
\begin{prop}
Let $A$ be an operator algebra, and let $B$ be an operator space such that the multiplication on $B$ is given by
\[b_1b_2 = \phi(b_1)b_2 (b_1; b_2 \in B) ;\]
where $\phi$ is as above. Then $A\ot^{h} B$
is Arens regular.
\end{prop}
\begin{pf}
Consider the left slice map $L_{\phi}: A\ot^{h} B \to A$ ($a\ot b\to \phi(b)a$). Note that $L_{\phi}$ is an algebra homomorphism. We claim
that $L_\phi$ is a bijective map. Fix $b_0\in B$ such that $\phi(b_0)=1$.
Thus for any $a\in A$, $L_{\phi}(a\ot b_0)=a$. Now let  $L_{\phi}(\ds\sum_{i=1}^{\infty}a_i\ot b_i)=0$.
As  $\phi$ is one-to-one, so $\ds\sum_{i=1}^{\infty}a_i\psi(b_i)=0$ for any $\psi\in B^*$ and hence $\ds\sum_{i=1}^{\infty}a_i\ot b_i=0$ by
(\cite{vandee3}, Proposition 4.4). Thus $L_{\phi}$ is a bijective map. Now let $m $ be a completely bounded bilinear form such that
 for any two pairs of sequences $(a_{i})$, $(b_{j})$ in $A_{1}$ and $(c_{i})$, $(d_{j})$ in $B_{1}$, the
 iterated limits $\displaystyle\lim_{i}\displaystyle\lim_{j} m(a_{i}b_{j}, c_{i}d_{j})$ and $\displaystyle\lim_{j}\displaystyle\lim_{i}m
 (a_{i}b_{j}, c_{i}d_{j})$ exist. Let $f$ be the associated linear functional corresponding to $m$. It is clear that $L_{\phi}(a_{i}\ot c_{i})$ and
 $L_{\phi}(b_{j}\ot d_{j})$ are in $A_1$. Using the fact that the operator algebras are  Arens regular, we have the following equality:
  \begin{align}
\displaystyle\lim_{i}\displaystyle\lim_{j} f\circ L_{\phi}^{-1} (L_{\phi}(a_{i}\ot c_{i}) L_{\phi}(b_{j}\ot d_{j}))
&=\displaystyle\lim_{j}\displaystyle\lim_{i} f\circ L_{\phi}^{-1} (L_{\phi}(a_{i}\ot c_{i}) L_{\phi}(b_{j}\ot d_{j}))\notag
\end{align}
Thus, by the algebra homomorphism of $L_{\phi}$, we have
\begin{align}
\displaystyle\lim_{i}\displaystyle\lim_{j} f(a_{i}b_{j}\ot c_{i}d_{j})
&=\displaystyle\lim_{j}\displaystyle\lim_{i} f(a_{i}b_{j}\ot c_{i}d_{j})\notag
\end{align}
and hence the result follows from Theorem \ref{arin55}.
\end{pf}

By Proposition \ref{main3}, for exact operator algebras $A$ and $B$ for which multiplication on $B$ is defined by
$b_1b_2=\phi(b_1)b_2$, $A\widehat{\ot} B$ is Arens regular.

Let us consider the following multiplication defined on elementary tensor as:
\[(a\ot b)(c\ot d)=m(a,b)(c\ot d),\]
for $a,c\in A$ and $b,d\in B$, where $m$ is a bounded bilinear form on $A\times B$  with $\|m\|\leq 1$.

\begin{prop}
For any two Banach algebras $A$ and $B$, $A\otimes ^{\gamma} B$ with respect to the above multiplication is an Arens regular Banach algebra.
\end{prop}
\begin{proof}
For $a, c\in A$ and $b,d\in B$, $a\ot b f(c\ot d)=f((c\ot d)(a\ot b))= f(m(c,d) (a\ot b))= m(c,d)f(a\ot b)$ and $f a\ot b (c\ot d)=f((a\ot b)(c\ot d))
=m(a,b)f(c\ot d)$. Thus $a\ot b f=
f(a\ot b)\tilde{m}$ and $f a\ot b=m(a,b)f$, which further gives   $f F(a\ot b)= F(f a \ot b)=F(m(a,b) f)=m(a,b) F(f)=\tilde{m}(a\ot b) F(f)$ and
$Ff(a\ot b)=
 F( a \ot b f)=F( f(a\ot b)\tilde{m})= F(\tilde{m})f(a\ot b)$, where $\tilde{m}$
is a bounded linear functional on $A\otimes ^{\gamma} B$. So by linearity and continuity, it follows that  $f F=F(f) \tilde{m}$ and
$Ff =F(\tilde{m}) f$. Thus $F \square G(f)=F(Gf)= F(G(\tilde{m}) f)= G(\tilde{m}) F(f)$ and $F \lozenge G(f)=G(f F)= F(f) G(\tilde{m})$, and hence the result.
\end{proof}
Note that if we consider the jointly completely bounded bilinear forms and the completely bounded bilinear forms instead of bounded bilinear forms
in  the above, then similar result holds for the Haagerup tensor product and the operator space projective tensor product.\\
\noindent Another multiplication is given as follows:
\[(a\ot b)(c\ot d)=m(a,d)(b\ot c),\]
for $a,c\in A$ and $b,d\in B$, where $m$ is a bounded bilinear form on $A\times B$. Then one can easily see that for any two Banach algebras $A$ and $B$, $A\otimes ^{\gamma} B$ with respect to the above multiplication is an Arens regular Banach algebra if and
only if  $m$ is biregular.


\begin{thebibliography}{a}
\bibitem{sinc}Allen, S. D., Sinclair, A. M. and Smith, R. R., The ideal structure of the Haagerup tensor product of $C^{*}$-algebras, \textit{J. Reine Angew. Math.} 442  (1993), 111--148.
\bibitem{Arikan}  Arikan, N., Arens regularity and reflexivity, \textit{Quart. J. Math.} 32 (1981), 383--388.
\bibitem{blect} Blecher, D. P., Tensor products which do not preserve  operator algebras, \textit{Math. Proc.
Camb. Philos Soc.} 108 (1990), 395--403.
\bibitem{quob} Blecher, D. P. and  Merdy, L. C., On quotients of function algebras and operator algebra structures on $l_p$, \textit{J. Operator
theory} 34 (1995), 315--346.
\bibitem{yood} Civin, P. and Bertram, Y., The second conjugate space of a Banach algebra as an
algebra, \textit{ Pacific J. Math. } 11(1961), 847--870.

\bibitem{effr}Effros, E. G. and Ruan, Z. J., On approximation properties for operator spaces, \textit{International J. Math.} 1 (1990),  163--187.
\bibitem{effros} Effros, E. G. and Ruan, Z. J., Operator spaces, \textit{Claredon Press-Oxford}, 2000.
\bibitem{for1}  Forrest, B., Arens regularity and discrete groups, \textit{Pacific J. 	Math. } 151(1991), 217--227.
\bibitem{for2}  Forrest, B., Arens regularity and the $A_p(G)$ groups, \textit{Proc. Amer. Math. Soc.} 119 (1993), 595--598.
\bibitem{free} Freedman, W. and \"{U}lger, A., The Phillipis Properties, \textit{Proc. Amer. Math. Soc.} 128 (2000), 2137--2145.
\bibitem{groth} Grothendieck, A., Criteres de compacticite dans les espaces fonctionels generaux, \textit{Amer. J. Math.} 74(1972), 168--186.
\bibitem{musat} Haagerup,  U. and  Musat, M., The Effros-Ruan conjecture for bilinear forms on
$C^*$-algebras, \textit{Invent. Math.} 174 (2008), 139--163.

\bibitem{r3}  Jain, R. and  Kumar, A., Ideals in operator space projective tensor products of $C^{*}$-algebras,
\textit{J. Aust. Math. Soc.} 91 (2011), 275--288.
\bibitem{vandee3}Kumar, A. and Rajpal, V., Projective tensor product of $C^{*}$-algebras, \textit{Adv. Pure Math. }4
(2014), 176--188.
\bibitem{ALAU} Lau, A. T.-M. and  Wong, J. C. S., Weakly almost periodic elements in $L^\infty(G)$
of a locally compact group $G$, \textit{Proc. Amer. Math. Soc.} 107 (1989), 1031-1036.
\bibitem{ulgerr} Lau, A. T.-M. and \"{U}lger, A., Some geometric properties on the Fourier and Fourier Stieltjes algebras of
 Locally Compact Groups, Arens Regularity and related problems, \textit{Trans. Amer. Math.
Soc. } 337(1993), 321--359.
\bibitem{shl} Pisier, G. and  Shlyakhtenko, D.,  Grothendieck's theorem for operator spaces, \textit{Invent. Math.} 150 (2002), 185--217.
\bibitem{pisi} Pisier, G., Introduction to operator space theory, \textit{Cambridge University Press}, 2003.
\bibitem{pym}Pym, J.S., The convolution of functionals on spaces of bounded functions, \textit{Proc. London Math. Soc.}
15 (1965), 84--104.
 \bibitem{vandee4}  Rajpal, V.,  Kumar, A. and  Itoh,T., Schur tensor product of operator spaces, \textit{To appear in Forum Mathematicum}, DOI: 10.1515/forum-2013-0142,
Available on arXiv:1308.4538v1 [math.OA].
\bibitem{ulger}\"{U}lger, A., Arens regularity of the algebra $A\otimes_{\gamma} B$, \textit{Trans. Amer. Math. Soc.} 305 (1988), 623--639.
\bibitem{ulger2nd}\"{U}lger, A., Arens regularity sometimes implies the RNP, \textit{Pacific  J.  Math.} 143 (1990), 377-399.
\bibitem{error} \"{U}lger, A., Erratum to Arens regularity of the algebra $A\otimes_{\gamma} B$,
\textit{Trans. Amer. Math. Soc.} 355 (2003), 3839.
\end{thebibliography}
\end{document}